\documentclass[11pt,a4paper]{article}
\usepackage[top=3cm, bottom=3cm, left=3cm, right=3cm]{geometry}
\usepackage{tabu}
\usepackage{breqn}
\usepackage[hang,flushmargin]{footmisc} 

\usepackage{amsmath,amsthm,amssymb,graphicx, multicol, array}
\usepackage{enumerate}
\usepackage{enumitem}
\usepackage{hyperref}
\usepackage{setspace}
\usepackage{xcolor}
\usepackage{currfile}

\usepackage{tabto}

\newtheorem{thm}{Theorem}[section]
\newtheorem{lem}{Lemma}[section]

\newcommand{\N}{\mathbb{N}}
\newcommand{\Z}{\mathbb{Z}}

\newcommand{\C}{\mathbb{C}}

\usepackage{fancyhdr}
\title{Average Orders of the Euler Phi Function, The Dedekind Psi Function, The Sum of Divisors Function, And The Largest Integer Function}
\date{}
\author{N. A. Carella}

\begin{document}
\thispagestyle{empty}
\date{}

\maketitle
\textbf{\textit{Abstract}:} Let $ x\geq 1 $ be a large number, let $ [x]=x-\{x\} $ be the largest integer function, and let $ \varphi(n)$ be the Euler totient function. The result $ \sum_{n\leq x}\varphi([x/n])=(6/\pi^2)x\log x+O\left (  x(\log x)^{2/3}(\log\log  x)^{1/3}\right ) $ was proved very recently. This note presents a short elementary proof, and sharpen the error term to $ \sum_{n\leq x}\varphi([x/n])=(6/\pi^2)x\log x+O(x) $. In addition, the first proofs of the asymptotics formulas for the finite sums $ \sum_{n\leq x}\psi([x/n])=(15/\pi^2)x\log x+O(x\log \log x) $, and $ \sum_{n\leq x}\sigma([x/n])=(\pi^2/6)x\log x+O(x \log \log x) $ are also evaluated here. 
\let\thefootnote\relax\footnote{\today  
}

\tableofcontents
\newpage
\section{Introduction} \label{s0800}
Some new analytic techniques for evaluating the \textit{fractional} finite sums $\sum_{n\leq x}f\left ([x/n]\right )$ for slow growing functions $ f(n)\ll n^{\varepsilon} $, were recently introduced in \cite{BS18}. Subsequently, the some analytic techniques for faster growing functions were introduced in the more recent literature as \cite{ZW20}. In this note, the standard analytic techniques originally developed for evaluating the average orders $ \sum_{n\leq x}f(n) $ of arithmetic functions are modified to handle the fractional finite sums $ \sum_{n\leq x}f([x/n]) $ for multiplicative functions defined by Dirichlet convolutions $f(n)=\sum_{d\mid n}g(n/d)$, where $f,g: \N\longrightarrow \C$, and the fast rates of growth approximately $ f(n)\gg n(\log n)^b $, for some $ b\in \Z $. The modified standard techniques are simpler, more efficient and produce very short proofs. \\

As demonstrations, the fractional finite sum $ \sum_{n\leq x}\varphi([x/n]) $ of the Euler phi function $  \varphi$ in Theorem \ref{thm0800.001}, the fractional finite sum $ \sum_{n\leq x}\psi([x/n]) $ of the Dedekind psi function $  \psi$ in Theorem \ref{thm0300.010}, and the fractional finite sum $ \sum_{n\leq x}\sigma([x/n]) $ of the sum of divisors function $  \sigma$ in Theorem \ref{thm0500.010}, are evaluated here. The three functions $ \varphi(n)\leq\psi(n)\leq \sigma(n) $ share many similarities such as multiplicative structures, rates of growths, are defined by some Dirichlet convolutions $f(n)=\sum_{d\mid n}g(n/d)$, et cetera, and have similar proofs. Theorem \ref{thm0800.001} has a very short proof, and sharpen the error term of a very recent result $\sum_{n\leq x}\varphi\left ([x/n]\right )=(6/\pi^2)x\log x+O\left (  x(\log x)^{2/3}(\log\log  x)^{1/3}\right )$ proved in \cite{ZW20} using a very complicated and lengthy proof. Further, Theorem \ref{thm0300.010}, and Theorem \ref{thm0500.010} are new results in the literature.

\section{Euler Totient Function}\label{s0801}
The Euler totient function $ \varphi(n)=n\sum_{d\mid n}\mu(d)/d $ is multiplicative and satisfies the growth condition $ \varphi(n)\gg n/\log \log n $. A very short proof for $\sum_{n\leq x}\varphi([x/n] )  $ is produced here. It is a modified version of the standard proof for the average order $ \sum_{n\leq x}\varphi(n)= (3/\pi^2)x^2+O\left (  x\log x\right )$, which appears in \cite[Theorem 3.7]{AP76}, and similar references.\\
 
\begin{thm}\label{thm0800.001} If $ x\geq 1 $ is a large number, then, 
\begin{equation}\label{eq0800.001}
\sum_{n\leq x}\varphi\left (\left [\frac{x}{n}\right ]\right )= \frac{6}{\pi^2}x\log x+O\left (   x\right ).
\end{equation}
\end{thm}
\begin{proof} Use the identity $ \varphi(n)=n\sum_{d\mid n}\mu(d)/d $ to rewrite the finite sum, and switch the order of summation:
\begin{eqnarray}\label{eq0800.003}
\sum_{n\leq x}\varphi\left (\left [\frac{x}{n}\right ]\right )
&=& \sum_{n\leq x} \left [\frac{x}{n}\right]\sum_{d\,\mid\, [x/n]}\frac{\mu(d)}{d} \\
&=& \sum_{d\leq x} \frac{\mu(d)}{d}\sum_{\substack{n\leq x\\d\,\mid\, [x/n]}}\left [\frac{x}{n}\right] \nonumber.
\end{eqnarray}

Apply Lemma \ref{lem0802.050} to remove the congruence on the inner sum index, and break it up into two subsums. Specifically,
\begin{eqnarray}\label{eq0800.015}
\sum_{d\leq x} \frac{\mu(d)}{d}\sum_{\substack{n\leq x\\d\,\mid\, [x/n]}}\left [\frac{x}{n}\right]
&=&\sum_{d\leq x} \frac{\mu(d)}{d}\sum_{n\leq x}\left [\frac{x}{n}\right]\cdot \frac{1}{d}\sum_{0\leq a\leq d-1}e^{i2\pi a[x/n]/d} \\
&=&\sum_{d\leq x} \frac{\mu(d)}{d^2}\sum_{n\leq x}\left [\frac{x}{n}\right]+\sum_{d\leq x} \frac{\mu(d)}{d^2}\sum_{n\leq x}\left [\frac{x}{n}\right]\sum_{0< a\leq d-1}e^{i2\pi a[x/n]/d}\nonumber \\
&=&M(x)\quad +\quad E(x)\nonumber.
\end{eqnarray}
The main term $M(x)$ is computed in Lemma \ref{lem0802.100} and the error term $E(x)$ is computed in Lemma \ref{lem0802.200}. Summing these expressions complete the verification.
\end{proof}

It is easy to verify that the expressions $ M(x) $ and $ E(x) $ imply the omega result
\begin{equation}\label{eq0800.014}
\sum_{n\leq x}\varphi\left (\left [\frac{x}{n}\right ]\right )-\frac{6}{\pi^2}x\log x=\Omega_{\pm}(x)
\end{equation}
or a better result.

\section{Auxiliary Results for the Phi Function}\label{s0802}
The detailed and elementary proofs of the preliminary results required in the proof of Theorem \ref{thm0800.001} concerning the Euler phi function $\varphi(n)=\sum_{d\mid n}\mu(d)d$ are recorded in this section.\\

\begin{lem}\label{lem0802.050} Let $ x\geq 1 $ be a large number, and let $1\leq d,n\leq x$ be integers. Then,
\begin{equation}\label{eq0802.010}
\frac{1}{d}\sum_{0\leq a\leq d-1}e^{i2\pi a[x/n]/d} =\left \{\begin{array}{ll}
1 & \text{ if } d\mid [x/n],  \\
0& \text{ if } d\nmid [x/n], \\
\end{array} \right .
\end{equation}
\end{lem}
\begin{proof} The two cases $ d\mid [x/n] $ and $ d\nmid [x/n] $ are easily handled with the basic exponential sum
\begin{equation}\label{eq0802.020}
\sum_{0\leq k\leq q-1}e^{i2\pi km/q} =\left \{\begin{array}{ll}
q & \text{ if } q\mid m,  \\
0& \text{ if } q\nmid m, \\
\end{array} \right .
\end{equation}
where $ m\ne0$, and $q\geq 1  $ are integers.
\end{proof}

\begin{lem}\label{lem0802.100} Let $ x\geq 1 $ be a large number. Then,
\begin{equation}\label{eq0802.110}
\sum_{d\leq x} \frac{\mu(d)}{d^2}\sum_{n\leq x}\left [\frac{x}{n}\right]
= \frac{6}{\pi^2}x\log x+O\left (x\right ) .
\end{equation}
\end{lem}
\begin{proof}Expand the bracket and evaluate the two subsums. Specifically,
\begin{eqnarray}\label{eq0802.115}
\sum_{d\leq x} \frac{\mu(d)}{d^2}\sum_{n\leq x}\left [\frac{x}{n}\right]&=&x\sum_{d\leq x} \frac{\mu(d)}{d^2}\sum_{n\leq x}\frac{1}{n}-\sum_{d\leq x} \frac{\mu(d)}{d^2}\sum_{n\leq x}\left \{\frac{x}{n}\right\}\\
&=&\frac{6}{\pi^2}x\log x+O\left (x\right )\nonumber 
\end{eqnarray}
as claimed
\end{proof}

\begin{lem}\label{lem0802.200} Let $ x\geq 1 $ be a large number. Then,
\begin{equation}\label{eq0802.210}
\sum_{d\leq x} \frac{\mu(d)}{d^2}\sum_{n\leq x}\left [\frac{x}{n}\right]\sum_{0< a\leq d-1}e^{i2\pi a[x/n]/d}
= O\left (x  \right ).
\end{equation}
\end{lem}
\begin{proof} As $n\leq x$ ranges over all its values, each value $m=[x/n]\leq x$ is repeated 
\begin{equation}\label{eq0802.215}
\left [\frac{x}{n}\right]-\left [\frac{x}{n+1}\right]=\frac{x}{n(n+1)}+O(1)
\end{equation}
times. Hence, reordering the triple sum $E(x)$, and substituting \eqref{eq0802.215} yield
\begin{eqnarray}\label{eq0802.220}
E(x)&=&\sum_{n\leq x}\left [\frac{x}{n}\right]\sum_{d\leq x} \frac{\mu(d)}{d^2}\sum_{0< a\leq d-1}e^{i2\pi a[x/n]/d}\\
&=&\sum_{n\leq x}\left (\frac{x}{n(n+1)}+O(1)\right)\sum_{d\leq x} \frac{\mu(d)}{d^2}\sum_{0< a\leq d-1}e^{i2\pi am/d}\nonumber\\
&=&S_0(x)+S_1(x)\nonumber.
\end{eqnarray}
The finite subsums $S_0(x)$, and $S_1(x)$ correspond to the subsets of integers $n\leq x$ such that $d\mid [x/n]$, and $d\nmid [x/n]$, respectively. \\

\textit{Case} I. The set of values $m=[x/n]\leq x$ such that $d\mid m$. Evaluating the indicator function returns
\begin{eqnarray}\label{eq0802.225}
S_0(x)&=&\sum_{n\leq x}\left (\frac{x}{n(n+1)}+O(1)\right)\sum_{\substack{d\leq x\\d\mid m}} \frac{\mu(d)}{d^2}\sum_{0< a\leq d-1}e^{i2\pi am/d}\\
&=&\sum_{n\leq x}\left (\frac{x}{n(n+1)}+O(1)\right)\sum_{\substack{d\leq x\\d\mid m}} \frac{\mu(d)}{d^2}\cdot (d-1)\nonumber\\
&=&\sum_{n\leq x}\left (\frac{x}{n(n+1)}+O(1)\right)\left (-\sum_{\substack{d\leq x\\d\mid m}} \frac{\mu(d)}{d^2}+\sum_{\substack{d\leq x\\d\mid m}} \frac{\mu(d)}{d}\right)\nonumber.
\end{eqnarray}
The first sum from  right in \eqref{eq0802.225} has the upper bound
\begin{equation}\label{eq0802.230}
\sum_{\substack{d\leq x\\d\mid m}} \frac{\mu(d)}{d}=O(1),
\end{equation}
this follows from \cite[Theorem 5]{MV02} and partial summation, or other summation methods. And the second sum from the right in \eqref{eq0802.225} is represented by
\begin{equation}\label{eq0802.235}
\sum_{\substack{d\leq x\\d\mid m}} \frac{\mu(d)}{d^2}=c_{0}(n)+O(1),
\end{equation}
where $|c_{0}(n)|<2$ depends on $n$. Replacing these expressions in \eqref{eq0802.225} returns
\begin{eqnarray}\label{eq0802.240}
S_0(x)&=&\sum_{n\leq x}\left (\frac{x}{n(n+1)}+O(1)\right)\left (c_0(n)+O(1)\right)\nonumber\\
&=&O\left (x\right)\nonumber.
\end{eqnarray}

\textit{Case} II. The set of values $m=[x/n]\leq x$ such that $d\nmid m$. Evaluating the indicator function returns
\begin{eqnarray}\label{eq0802.245}
S_1(x)&=&\sum_{n\leq x}\left (\frac{x}{n(n+1)}+O(1)\right)\sum_{\substack{d\leq x\\d\nmid m}} \frac{\mu(d)}{d^2}\sum_{0< a\leq d-1}e^{i2\pi am/d}\\
&=&\sum_{n\leq x}\left (\frac{x}{n(n+1)}+O(1)\right)\sum_{\substack{d\leq x\\d\nmid m}} \frac{\mu(d)}{d^2}\cdot (-1)\nonumber\\
&=&\sum_{n\leq x}\left (\frac{x}{n(n+1)}+O(1)\right)\left (c_1(n)+O(1)\right)\nonumber\\
&=&O\left (x\right)\nonumber,
\end{eqnarray}
where $|c_1(n)|<2$ depends on $n$.\\

Summing the last two expressions yield $E(x)=S_0(x)+S_1(x)=O\left (x\right)$.   
\end{proof}

\section{Dedekind Psi Function}\label{s0300}
The second result deals with the Dedekind function $ \psi(n)=n\sum_{d\mid n}\mu(d)^2/d $. It is multiplicative and satisfies the growth condition $ \psi(n)\gg n $. The first asymptotic formula for the fractional finite sum of the Dedekind function is given below.
\begin{thm}\label{thm0300.010} If $ x\geq 1 $ is a large number, then, 
\begin{equation}\label{eq0300.010}
\sum_{n\leq x}\psi\left (\left [\frac{x}{n}\right ]\right )= \frac{15}{\pi^2}x\log x+O\left ( x\log \log x\right ).
\end{equation}
\end{thm}
\begin{proof} Use the identity $ \psi(n)=n\sum_{d\mid n}\mu^2(d)/d $ to rewrite the finite sum, and switch the order of summation:
\begin{eqnarray}\label{eq0300.013}
\sum_{n\leq x}\psi\left (\left [\frac{x}{n}\right ]\right )
&=& \sum_{n\leq x} \left [\frac{x}{n}\right]\sum_{d\,\mid\, [x/n]}\frac{\mu^2(d)}{d} \\
&=& \sum_{d\leq x} \frac{\mu^2(d)}{d}\sum_{\substack{n\leq x\\d\,\mid\, [x/n]}}\left [\frac{x}{n}\right] \nonumber.
\end{eqnarray}
Apply Lemma \ref{lem0802.050} to remove the congruence on the inner sum index, and break it up into two subsums. Specifically,
\begin{eqnarray}\label{eq0300.015}
\sum_{d\leq x} \frac{\mu^2(d)}{d}\sum_{\substack{n\leq x\\d\,\mid\, [x/n]}}\left [\frac{x}{n}\right]
&=&\sum_{d\leq x} \frac{\mu^2(d)}{d}\sum_{n\leq x}\left [\frac{x}{n}\right]\cdot \frac{1}{d}\sum_{0\leq a\leq d-1}e^{i2\pi a[x/n]/d} \\
&=&\sum_{d\leq x} \frac{\mu^2(d)}{d^2}\sum_{n\leq x}\left [\frac{x}{n}\right]+\sum_{d\leq x} \frac{\mu^2(d)}{d^2}\sum_{n\leq x}\left [\frac{x}{n}\right]\sum_{0< a\leq d-1}e^{i2\pi a[x/n]/d}\nonumber \\
&=&M_1(x)\quad +\quad E_1(x)\nonumber.
\end{eqnarray}
The main term $M_1(x)$ is computed in Lemma \ref{lem0402.100} and the error term $E_1(x)$ is computed in Lemma \ref{lem0402.200}. Summing these expressions complete the verification.
\end{proof}

It is easy to verify that the subsums $ M_1(x) $ and $ E_1(x) $ imply the omega result
\begin{equation}\label{eq0300.014}
\sum_{n\leq x}\psi\left (\left [\frac{x}{n}\right ]\right )-\frac{15}{\pi^2}x\log x=\Omega_{\pm}(x)
\end{equation}
or a better result. A sketch of the standard proof for the average order
\begin{equation}\label{eq0300.016}
\sum_{n\leq x}\psi(n)= \frac{15}{\pi^2}x^2+O\left (  x\log x\right ),
\end{equation}
appears in \cite[Exercise 13, p.\ 71]{AP76}.

\section{Auxiliary Results for the Psi Function}\label{s0402}
The detailed and elementary proofs of the preliminary results required in the proof of Theorem \ref{thm0300.010} concerning the Dedekind psi function $\psi(n)=\sum_{d\mid n}\mu^2(d)d$ are recorded in this section.

\begin{lem}\label{lem0402.100} Let $ x\geq 1 $ be a large number. Then,
\begin{equation}\label{eq0402.110}
\sum_{d\leq x} \frac{\mu^2(d)}{d^2}\sum_{n\leq x}\left [\frac{x}{n}\right]
= \frac{15}{\pi^2}x\log x+O\left (x\right ) .
\end{equation}
\end{lem}
\begin{proof}Expand the bracket and evaluate the two subsums. Specifically,
\begin{eqnarray}\label{eq0402.115}
\sum_{d\leq x} \frac{\mu^2(d)}{d^2}\sum_{n\leq x}\left [\frac{x}{n}\right]&=&x\sum_{d\leq x} \frac{\mu^2(d)}{d^2}\sum_{n\leq x}\frac{1}{n}-\sum_{d\leq x} \frac{\mu^2(d)}{d^2}\sum_{n\leq x}\left \{\frac{x}{n}\right\}\\
&=&\frac{15}{\pi^2}x\log x+O\left (x\right )\nonumber 
\end{eqnarray}
as claimed
\end{proof}

\begin{lem}\label{lem0402.200} Let $ x\geq 1 $ be a large number. Then,
\begin{equation}\label{eq0402.210}
\sum_{d\leq x} \frac{\mu^2(d)}{d^2}\sum_{n\leq x}\left [\frac{x}{n}\right]\sum_{0< a\leq d-1}e^{i2\pi a[x/n]/d}
= O\left (x  \log \log x\right ).
\end{equation}
\end{lem}
\begin{proof} As $n\leq x$ ranges over all its values, each value $m=[x/n]\leq x$ is repeated 
\begin{equation}\label{eq0402.215}
\left [\frac{x}{n}\right]-\left [\frac{x}{n+1}\right]=\frac{x}{n(n+1)}+O(1)
\end{equation}
times. Hence, reordering the triple sum $E(x)$, and substituting \eqref{eq0402.215} yield
\begin{eqnarray}\label{eq0402.220}
E_1(x)&=&\sum_{n\leq x}\left [\frac{x}{n}\right]\sum_{d\leq x} \frac{\mu^2(d)}{d^2}\sum_{0< a\leq d-1}e^{i2\pi a[x/n]/d}\\
&=&\sum_{n\leq x}\left (\frac{x}{n(n+1)}+O(1)\right)\sum_{d\leq x} \frac{\mu^2(d)}{d^2}\sum_{0< a\leq d-1}e^{i2\pi am/d}\nonumber\\
&=&S_{10}(x)+S_{11}(x)\nonumber.
\end{eqnarray}
The finite subsums $S_{10}(x)$, and $S_{11}(x)$ correspond to the subsets of integers $n\leq x$ such that $d\mid [x/n]$, and $d\nmid [x/n]$, respectively. \\

\textit{Case} I. The set of values $m=[x/n]\leq x$ such that $d\mid m$. Evaluating the indicator function returns
\begin{eqnarray}\label{eq0402.225}
S_{10}(x)&=&\sum_{n\leq x}\left (\frac{x}{n(n+1)}+O(1)\right)\sum_{\substack{d\leq x\\d\nmid m}} \frac{\mu^2(d)}{d^2}\sum_{0< a\leq d-1}e^{i2\pi am/d}\\
&=&\sum_{n\leq x}\left (\frac{x}{n(n+1)}+O(1)\right)\sum_{\substack{d\leq x\\d\mid m}} \frac{\mu^2(d)}{d^2}\cdot (d-1)\nonumber\\
&=&\sum_{n\leq x}\left (\frac{x}{n(n+1)}+O(1)\right)\left (-\sum_{\substack{d\leq x\\d\mid m}} \frac{\mu^2(d)}{d^2}+\sum_{\substack{d\leq x\\d\mid m}} \frac{\mu^2(d)}{d}\right)\nonumber.
\end{eqnarray}
The first sum from  right in \eqref{eq0402.225} has the upper bound
\begin{equation}\label{eq0402.230}
\sum_{\substack{d\leq x\\d\mid m}} \frac{\mu^2(d)}{d}\leq \frac{\psi(m)}{m}=O(\log \log x),
\end{equation}
since $m\leq x$. And the second sum from the right in \eqref{eq0402.225} is represented by
\begin{equation}\label{eq0402.235}
\sum_{\substack{d\leq x\\d\mid m}} \frac{\mu(d)}{d^2}=c_{0}(n)+O(1),
\end{equation}
where $|c_{10}(n)|<2$ depends on $n$. Replacing these expressions in \eqref{eq0402.225} returns
\begin{eqnarray}\label{eq0402.240}
S_{10}(x)&=&\sum_{n\leq x}\left (\frac{x}{n(n+1)}+O(1)\right)\left (c_{10}(n)+O(\log \log x)\right)\\
&=&O\left (x\log \log x\right)\nonumber.
\end{eqnarray}

\textit{Case} II. The set of values $m=[x/n]\leq x$ such that $d\nmid m$. Evaluating the indicator function returns 
\begin{eqnarray}\label{eq0402.245}
S_{11}(x)&=&\sum_{n\leq x}\left (\frac{x}{n(n+1)}+O(1)\right)\sum_{\substack{d\leq x\\d\nmid m}} \frac{\mu^2(d)}{d^2}\sum_{0< a\leq d-1}e^{i2\pi am/d}\\
&=&\sum_{n\leq x}\left (\frac{x}{n(n+1)}+O(1)\right)\sum_{\substack{d\leq x\\d\nmid m}} \frac{\mu^2(d)}{d^2}\cdot (-1)\nonumber\\
&=&\sum_{n\leq x}\left (\frac{x}{n(n+1)}+O(1)\right)\left (c_{11}(n)+O(1)\right)\nonumber\\
&=&O\left (x\right)\nonumber,
\end{eqnarray}
where $|c_{11}(n)|<2$ depends on $n$.\\

Summing the last two expressions yield $E_1(x)=S_{10}(x)+S_{11}(x)=O\left (x\log \log x\right)$.   
\end{proof}

\section{Sum of Divisors Function}\label{s0500}
The third result deals with the sum of divisors function $ \sigma(n)=n\sum_{d\mid n}1/d $. It is multiplicative and satisfies the growth condition $ \sigma(n)\gg n $. The first asymptotic formula for the fractional sum of divisor function is given below.
\begin{thm}\label{thm0500.010} If $ x\geq 1 $ is a large number, then, 
\begin{equation}\label{eq0500.010}
\sum_{n\leq x}\sigma\left (\left [\frac{x}{n}\right ]\right )= \frac{\pi^2}{6}x\log x+O\left ( x\log \log x\right ).
\end{equation}
\end{thm}
\begin{proof} Use the identity $ \sigma(n)=n\sum_{d\mid n}1/d $ to rewrite the finite sum, and switch the order of summation:
\begin{eqnarray}\label{eq0500.003}
\sum_{n\leq x}\sigma\left (\left [\frac{x}{n}\right ]\right )
&=& \sum_{n\leq x} \left [\frac{x}{n}\right]\sum_{d\,\mid\, [x/n]}\frac{1}{d} \\
&=& \sum_{d\leq x} \frac{1}{d}\sum_{\substack{n\leq x\\d\,\mid\, [x/n]}}\left [\frac{x}{n}\right] \nonumber.
\end{eqnarray}
Apply Lemma \ref{lem0802.050} to remove the congruence on the inner sum index, and break it up into two subsums. Specifically,
\begin{eqnarray}\label{eq0500.015}
\sum_{d\leq x} \frac{1}{d}\sum_{\substack{n\leq x\\d\,\mid\, [x/n]}}\left [\frac{x}{n}\right]
&=&\sum_{d\leq x} \frac{1}{d}\sum_{n\leq x}\left [\frac{x}{n}\right]\cdot \frac{1}{d}\sum_{0\leq a\leq d-1}e^{i2\pi a[x/n]/d} \\
&=&\sum_{d\leq x} \frac{1}{d^2}\sum_{n\leq x}\left [\frac{x}{n}\right]+\sum_{d\leq x} \frac{1}{d^2}\sum_{n\leq x}\left [\frac{x}{n}\right]\sum_{0< a\leq d-1}e^{i2\pi a[x/n]/d}\nonumber \\
&=&M_2(x)\quad +\quad E_2(x)\nonumber.
\end{eqnarray}
The main term $M_2(x)$ is computed in Lemma \ref{lem0502.100} and the error term $E_2(x)$ is computed in Lemma \ref{lem0502.200}. Summing these expressions complete the verification.
\end{proof}

It is easy to verify that the subsums $ M_2(x) $ and $ E_2(x) $ imply the omega result
\begin{equation}\label{eq0800.014}
\sum_{n\leq x}\sigma\left (\left [\frac{x}{n}\right ]\right )-\frac{\pi^2}{6}x\log x=\Omega_{\pm}(x)
\end{equation}
or a better result. The standard proof for the average order
\begin{equation}\label{eq0800.016}
\sum_{n\leq x}\sigma(n)= \frac{\pi^2}{12}x^2+O\left (  x\log x\right ),
\end{equation}
appears in \cite[Theorem 3.4]{AP76}.

\section{Auxiliary Results for the Sigma Function}\label{s0502}
The detailed and elementary proofs of the preliminary results required in the proof of Theorem \ref{thm0500.010} concerning the sum of divisor function $\sigma(n)=\sum_{d\mid n}d$ are recorded in this section.

\begin{lem}\label{lem0502.100} Let $ x\geq 1 $ be a large number. Then,
\begin{equation}\label{eq0502.110}
\sum_{d\leq x} \frac{1}{d^2}\sum_{n\leq x}\left [\frac{x}{n}\right]
= \frac{\pi^2}{6}x\log x+O\left (x\right ) .
\end{equation}
\end{lem}
\begin{proof}Expand the bracket and evaluate the two subsums. Specifically,
\begin{eqnarray}\label{eq0502.115}
\sum_{d\leq x} \frac{1}{d^2}\sum_{n\leq x}\left [\frac{x}{n}\right]&=&x\sum_{d\leq x} \frac{1}{d^2}\sum_{n\leq x}\frac{1}{n}-\sum_{d\leq x} \frac{1}{d^2}\sum_{n\leq x}\left \{\frac{x}{n}\right\}\\
&=&\frac{\pi^2}{6}x\log x+O\left (x\right )\nonumber 
\end{eqnarray}
as claimed
\end{proof}

\begin{lem}\label{lem0502.200} Let $ x\geq 1 $ be a large number. Then,
\begin{equation}\label{eq0502.210}
\sum_{d\leq x} \frac{1}{d^2}\sum_{n\leq x}\left [\frac{x}{n}\right]\sum_{0< a\leq d-1}e^{i2\pi a[x/n]/d}
= O\left (x  \log \log x\right ).
\end{equation}
\end{lem}
\begin{proof} As $n\leq x$ ranges over all its values, each value $m=[x/n]\leq x$ is repeated 
\begin{equation}\label{eq0502.215}
\left [\frac{x}{n}\right]-\left [\frac{x}{n+1}\right]=\frac{x}{n(n+1)}+O(1)
\end{equation}
times. Hence, reordering the triple sum $E_2(x)$, and substituting \eqref{eq0502.215} yield
\begin{eqnarray}\label{eq0502.220}
E_2(x)&=&\sum_{n\leq x}\left [\frac{x}{n}\right]\sum_{d\leq x} \frac{1}{d^2}\sum_{0< a\leq d-1}e^{i2\pi a[x/n]/d}\\
&=&\sum_{n\leq x}\left (\frac{x}{n(n+1)}+O(1)\right)\sum_{d\leq x} \frac{1}{d^2}\sum_{0< a\leq d-1}e^{i2\pi am/d}\nonumber\\
&=&S_{20}(x)+S_{21}(x)\nonumber.
\end{eqnarray}
The finite subsums $S_{20}(x)$, and $S_{21}(x)$ correspond to the subsets of integers $n\leq x$ such that $d\mid [x/n]$, and $d\nmid [x/n]$, respectively. \\

\textit{Case} I. The set of values $m=[x/n]\leq x$ such that $d\mid m$. Evaluating the indicator function returns
\begin{eqnarray}\label{eq0502.225}
S_{20}(x)&=&\sum_{n\leq x}\left (\frac{x}{n(n+1)}+O(1)\right)\sum_{\substack{d\leq x\\d\mid m}} \frac{1}{d^2}\sum_{0< a\leq d-1}e^{i2\pi am/d}\\
&=&\sum_{n\leq x}\left (\frac{x}{n(n+1)}+O(1)\right)\sum_{\substack{d\leq x\\d\mid m}} \frac{1}{d^2}\cdot (d-1)\nonumber\\
&=&\sum_{n\leq x}\left (\frac{x}{n(n+1)}+O(1)\right)\left (-\sum_{\substack{d\leq x\\d\mid m}} \frac{1}{d^2}+\sum_{\substack{d\leq x\\d\mid m}} \frac{1}{d}\right)\nonumber.
\end{eqnarray}
The first sum from the right in \eqref{eq0502.225} has the upper bound
\begin{equation}\label{eq0502.230}
\sum_{\substack{d\leq x\\d\mid m}} \frac{1}{d}\leq \frac{\sigma(m)}{m}=O(\log \log x),
\end{equation}
since $m\leq x$. And the second sum from the right in \eqref{eq0502.225} is represented by 
\begin{equation}\label{eq0502.235}
\sum_{\substack{d\leq x\\d\mid m}} \frac{\mu(d)}{d^2}=c_{0}(n)+O(1),
\end{equation}
where $|c_{20}(n)|<2$ depends on $n$. Replacing these expressions in \eqref{eq0502.225} returns
\begin{eqnarray}\label{eq0502.240}
S_{20}(x)
&=&\sum_{n\leq x}\left (\frac{x}{n(n+1)}+O(1)\right)\left (c_{10}(n)+O(\log \log x)\right)\\
&=&O\left (x\log \log x\right)\nonumber.
\end{eqnarray}

\textit{Case} II. The set of values $m=[x/n]\leq x$ such that $d\nmid m$. Evaluating the indicator function returns
\begin{eqnarray}\label{eq0502.245}
S_{21}(x)&=&\sum_{n\leq x}\left (\frac{x}{n(n+1)}+O(1)\right)\sum_{\substack{d\leq x\\d\nmid m}} \frac{1}{d^2}\sum_{0< a\leq d-1}e^{i2\pi am/d}\\
&=&\sum_{n\leq x}\left (\frac{x}{n(n+1)}+O(1)\right)\sum_{\substack{d\leq x\\d\nmid m}} \frac{1}{d^2}\cdot (-1)\nonumber\\
&=&\sum_{n\leq x}\left (\frac{x}{n(n+1)}+O(1)\right)\left (c_{21}(n)+O(1)\right)\nonumber\\
&=&O\left (x\right)\nonumber,
\end{eqnarray}
where $|c_{21}(n)|<2$ depends on $n$. \\ 

Summing the last two expressions yield $E_2(x)=S_{20}(x)+S_{21}(x)=O\left (x\log \log x\right)$.   
\end{proof}

\section{Numerical Data}\label{s0505}
Small numerical tables were generated by an online computer algebra system, the range of numbers $ x\leq 10^5 $ is limited by the wi-fi bandwidth. The error terms are defined by
\begin{equation}\label{eq0505.033}
E_1(x)=\sum_{n\leq x}\varphi\left (\left [\frac{x}{n}\right ]\right )- \frac{6}{\pi^2}x\log x,
\end{equation}
and  
\begin{equation}\label{eq0505.035}
E_2(x)=\sum_{n\leq x}\sigma\left (\left [\frac{x}{n}\right ]\right )- \frac{\pi^2}{6}x\log x,
\end{equation}
respectively. All the calculations are within the predicted ranges $ E_i(x)=O(x) $.
\begin{table}[h!]
\centering
\caption{Numerical Data For $\sum_{n\leq x}\varphi([x/n])$.} \label{t0505.003}
\begin{tabular}{l|l|r| r}
$x$&$\sum_{n\leq x}\varphi([x/n])$&$6\pi^{-2}x\log x$&Error $E_1(x)$\\
\hline
10&$17$&   $14.00$   &$3.00$\\
100&$275$&  $279.96$   &$-4.96$\\
1000&$4053$&   $4199.41$   &$146.41$\\
10000&$52201$&  $55992.16$   &$-3791.16$\\
100000&$673929$&   $699901.94$   &$-25972.94$\\
\end{tabular}
\end{table}

\begin{table}[h!]
\centering
\caption{Numerical Data For $\sum_{n\leq x}\sigma([x/n])$.} \label{t0505.009}
\begin{tabular}{l|l|r| r}
$x$&$\sum_{n\leq x}\sigma([x/n])$&$6^{-1}\pi^{2}x\log x$&Error $E_2(x)$\\
\hline
10&$39$&   $37.88$   &$1.12$\\
100&$804$&  $757.52$   &$46.48$\\
1000&$12077$&   $11362.80$   &$714.20$\\
10000&$167617$&  $151504.03$   &$16112.97$\\
100000&$2033577$&   $1893800.33$   &$139776.67$\\
\end{tabular}
\end{table}

\currfilename.\\

\end{document}